\documentclass[oneside,english]{amsart}
\usepackage[T1]{fontenc}
\usepackage[latin9]{inputenc}
\usepackage{geometry}
\usepackage{amsthm}
\usepackage{amssymb}

\makeatletter
\numberwithin{equation}{section}
\numberwithin{figure}{section}
 \theoremstyle{definition}
 \newtheorem*{defn*}{\protect\definitionname}
\theoremstyle{plain}
\newtheorem{thm}{\protect\theoremname}
  \theoremstyle{plain}
  \newtheorem{lem}[thm]{\protect\lemmaname}
  \theoremstyle{plain}
  \newtheorem{cor}[thm]{\protect\corollaryname}
  \theoremstyle{plain}
  \newtheorem{prop}[thm]{\protect\propositionname}
  \theoremstyle{remark}
  \newtheorem{rem}[thm]{\protect\remarkname}
  \theoremstyle{definition}
  \newtheorem{example}[thm]{\protect\examplename}

\usepackage{amsfonts}
\usepackage{xypic}
\newcommand{\A}{{\mathbb A}}

\makeatother

\usepackage{babel}
  \providecommand{\corollaryname}{Corollary}
  \providecommand{\definitionname}{Definition}
  \providecommand{\examplename}{Example}
  \providecommand{\lemmaname}{Lemma}
  \providecommand{\propositionname}{Proposition}
  \providecommand{\remarkname}{Remark}
\providecommand{\theoremname}{Theorem}

\begin{document}

\title[Cylinders in del Pezzo fibrations]{Cylinders in del Pezzo fibrations}

\author{Adrien Dubouloz}

\address{IMB UMR5584, CNRS, Univ. Bourgogne Franche-Comté, F-21000 Dijon,
France.}

\author{Takashi Kishimoto}

\email{adrien.dubouloz@u-bourgogne.fr}

\address{Department of Mathematics, Faculty of Science, Saitama University,
Saitama 338-8570, Japan}

\email{tkishimo@rimath.saitama-u.ac.jp}

\thanks{This project was partially funded by ANR Grant \textquotedbl{}BirPol\textquotedbl{}
ANR-11-JS01-004-01 and Grant-in-Aid for Scientific Research of JSPS
No. 15K04805. The research was done during a visit of the first author
at the University of Saitama, The authors thank this institution for
its generous support and the excellent working conditions offered. }


\begin{abstract}
We show that a del Pezzo fibration $\pi:V\rightarrow W$ of degre
$d$ contains a vertical open cylinder, that is, an open subset whose
intersection with the generic fiber of $\pi$ is isomorphic to $Z\times\mathbb{A}_{K}^{1}$
for some quasi-projective variety $Z$ defined over the function field
$K$ of $W$, if and only if $d\geq5$ and $\pi:V\rightarrow W$ admits
a rational section. We also construct twisted cylinders in total spaces
of threefold del Pezzo fibrations $\pi:V\rightarrow\mathbb{P}^{1}$
of degree $d\leq4$. 
\end{abstract}
\maketitle

\section*{Introduction}

An $\mathbb{A}_{k}^{r}$-\emph{cylinder} in a normal algebraic variety
$X$ defined over a field $k$ is a Zariski open subset $U$ isomorphic
to $Z\times\mathbb{A}_{k}^{r}$ for some algebraic variety $Z$ defined
over $k$. Complex projective varieties containing cylinders have
recently started to receive a lot of attention in connection with
the study of unipotent group actions on affine varieties. Namely,
it was established in \cite{KPZ1,KPZ2} that the existence of a nontrivial
action of the additive group $\mathbb{G}_{a,\mathbb{C}}$ on the affine
cone associated with a polarized projective variety $(X,H)$ is equivalent
to the existence in $X$ of a so-called $H$-polar cylinder, that
is, an $\mathbb{A}^{1}$-cylinder $U=X\setminus\mathrm{Supp}(D)$
for some effective divisor $D\in\left|mH\right|$, $m\geq1$. 

Since a complex projective variety $X$ containing a cylinder is in
particular birationally ruled, if it exists, the output $V$ of a
Minimal Model Program run on $X$ is necessarily is a Mori fiber space,
that is, a projective variety with $\mathbb{Q}$-factorial terminal
singularities equipped with an extremal contraction $\pi:V\rightarrow W$
over a lower dimensional normal projective variety $W$. In case where
$\dim W=0$, $V$ is a Fano variety of Picard number one. If $\dim V=2$
then $V$ isomorphic to $\mathbb{P}^{2}$ hence contains many (anti-canonically
polar) cylinders. In higher dimension, several families of examples
of Fano varieties of dimension $3$ and $4$ and Picard number one
admitting (anti-canonically polar) cylinders have been constructed
\cite{KPZ4,PZ14,PZ15}, but a complete classification is still far
from being known. 

In this article, we consider the question of existence of cylinders
in other possible outputs of Minimal Model Programs: del Pezzo fibrations
$\pi:V\rightarrow W$, which correspond to the case where $\dim W=\dim V-2$.
The general closed fibers of such fibrations are smooth del Pezzo
surfaces, the degree $\deg(V/W)$ of the del Pezzo fibration $\pi:V\rightarrow W$
is then defined as the degree of such a general fiber. Del Pezzo surfaces
contain many cylinders, indeed being isomorphic to $\mathbb{P}^{1}\times\mathbb{P}^{1}$
or the blow-up of $\mathbb{P}^{2}$ in at most eight points in general
position, they even have the property that every of their closed points
admits an open neighborhood isomorphic to $\mathbb{A}^{2}$. One could
therefore expect that given a del Pezzo fibration $\pi:V\rightarrow W$,
some suitably chosen families of ``fiber wise'' cylinders can be
arranged into a ``relative'' cylinder with respect to $\pi:V\rightarrow W$,
more precisely a so-called \emph{vertical cylinder}:
\begin{defn*}
Let $f:X\to Y$ be a morphism between normal algebraic varieties defined
over a field $k$ and let $U\simeq Z\times\mathbb{A}_{k}^{r}$ be
an $\mathbb{A}_{k}^{r}$-cylinder inside $X$. We say that $U$ is
\emph{vertical with respect to} $f$ if the restriction $f|_{U}$
factors as 
\[
f\mid_{U}=h\circ\mathrm{pr}_{Z}:U\simeq Z\times\mathbb{A}_{k}^{r}\stackrel{\mathrm{pr}_{Z}}{\longrightarrow}Z\stackrel{h}{\longrightarrow}Y
\]
for a suitable morphism $h:Z\rightarrow Y$. Otherwise, we say that
$U$ is a \emph{twisted} $\mathbb{A}_{k}^{r}$-\emph{cylinder} with
respect to $f$. 
\end{defn*}
For dominant morphisms $f:X\to Y$ the existence inside $X$ of an
$\mathbb{A}_{k}^{r}$-cylinder vertical with respect to $f$ translates
equivalently into that of an $\mathbb{A}_{K}^{r}$-cylinder inside
the fiber $X_{\eta}$ of $f$ over the generic point $\eta$ of $Y$,
considered as a variety defined over the function field $K$ of $Y$
(see Lemma \ref{lem:Vertical-gen} below). It follows in particular
that, for instance, a del Pezzo fibration $\pi:V\rightarrow W$ of
degree $9$ whose generic fiber $V_{\eta}$ is a Severi-Brauer surface
without rational point over the function field of $W$ cannot contain
any vertical $\mathbb{A}_{\mathbb{C}}^{1}$-cylinder. Our first main
result is a complete characterization of del Pezzo fibrations admitting
vertical cylinders:
\begin{thm}
\label{thm:dP-Main} A del Pezzo fibration $\pi:V\rightarrow W$ admits
a vertical $\mathbb{A}_{\mathbb{C}}^{1}$-cylinder if and only if
$\deg(V/W)\geq5$ and $\pi:V\rightarrow W$ has a rational section. 
\end{thm}
It follows in particular from this characterization that a del Pezzo
fibration fibration $\pi:V\rightarrow\mathbb{P}^{1}$ of degree $d\leq4$
does not admit any vertical $\mathbb{A}_{\mathbb{C}}^{1}$-cylinder.
But the reasons which prevent the existence of such vertical cylinders
does not give much insight concerning that of twisted cylinders in
total spaces of such fibrations. By general results due to Alekseev
\cite{Al87} and Pukhlikov \cite{Pu98}, ``most'' del Pezzo fibrations
$\pi:V\rightarrow\mathbb{P}^{1}$ of degree $d\leq4$ with smooth
total spaces are non-rational. On the other hand, since $\pi:V\rightarrow\mathbb{P}^{1}$
has a section by virtue of the Tsen-Lang Theorem, it follows from
\cite[Theorem 29.4]{Manin} that the total space of such a fibration
of degree $d\geq3$ is always unirational. As a consequence, ``most''
del Pezzo fibrations $\pi:V\rightarrow\mathbb{P}^{1}$ of degree $d=3,4$
with smooth total spaces cannot contain any $\mathbb{A}_{\mathbb{C}}^{1}$-cylinder
at all, vertical or twisted. Indeed if $V$ contains an open subset
$U\simeq Z\times\mathbb{A}_{\mathbb{C}}^{1}$ for some smooth quasi-projective
surface $Z$, then $Z$ is unirational hence rational, which would
imply in turn the rationality of $V$. 

Our second main result consists of a construction of families of del
Pezzo fibrations $\pi:V\rightarrow\mathbb{P}^{1}$ of any degree $d\leq4$
containing twisted cylinders of maximal possible dimension, which
arise as projective completions of $\mathbb{A}_{\mathbb{C}}^{3}$: 
\begin{thm}
\label{thm:MainThm2} For every $d\leq4$, there exist del Pezzo fibrations
$\pi:V\rightarrow\mathbb{P}^{1}$ of degree $d$ whose total spaces
contain $\mathbb{A}_{\mathbb{C}}^{3}$ as a twisted cylinder. 
\end{thm}
The scheme of the article is as follows. The first section is devoted
to the proof of Theorem \ref{thm:dP-Main}: we first establish in
subsection \ref{sub:Non-existence} that a del Pezzo fibration of
degree $d\leq4$ does not admit vertical $\mathbb{A}_{\mathbb{C}}^{1}$-cylinder,
then in subsection \ref{sub:cylinders-exist} we give explicit constructions
of $\mathbb{A}^{1}$-cylinders inside generic fibers of del Pezzo
fibrations $\pi:V\rightarrow W$ of degree $d\geq5$ with a rational
section. We also discuss as a complement in subsection \ref{sub:A2-cyl}
the existence of vertical $\mathbb{A}_{\mathbb{C}}^{2}$-cylinders
in del Pezzo fibrations $\pi:V\rightarrow W$ of degree $d\geq5$.
Then in section two, we first review the general setup for the construction
of projective completions of $\mathbb{A}_{\mathbb{C}}^{3}$ into total
spaces of del Pezzo fibrations $\pi:V\rightarrow\mathbb{P}^{1}$ established
in \cite{DK3}. Then we proceed in detail to the construction of such
completions for the specific case $d=4$, which was announced without
proof in \cite{DK3}.

\section{Vertical cylinders in del Pezzo fibrations}

Letting $\pi:V\rightarrow W$ be a del Pezzo fibration over a normal
projective variety $W$ with function field $K$, the following Lemma
\ref{lem:Vertical-gen} implies that the existence of an $\mathbb{A}^{1}$-cylinder
$U\subset V$ vertical with respect to $\pi$ is equivalent to that
of an $\mathbb{A}_{K}^{1}$-cylinder in the fiber $S$ of $\pi$ over
the generic point of $W$. On the other hand, the existence of a rational
section of $\pi:V\rightarrow W$ is equivalent to the existence of
a $K$-rational point of $S$. Since $S$ is a smooth del Pezzo surface
defined over $K$, of degree $\deg(V/W)$ and with Picard number $\rho_{K}(S)$
equal to one, to establish Theorem \ref{thm:dP-Main}, it is thus
enough to show that such a smooth del Pezzo surface admits an $\mathbb{A}_{K}^{1}$-cylinder
if and only if it has degree $d\geq5$ and a $K$-rational point.
\begin{lem}
\label{lem:Vertical-gen} Let $f:X\to Y$ be a dominant morphism between
normal algebraic varieties over a field $k$. Then $X$ contains a
vertical $\mathbb{A}_{k}^{r}$-cylinder with respect to $f$ if and
only if the generic fiber of $f$ contains an open subset of the form
$T\times\mathbb{A}_{K}^{r}$ for some algebraic variety $T$ defined
over the function field $K$ of $Y$. \end{lem}
\begin{proof}
Indeed, if $Z\times\mathbb{A}_{k}^{r}\simeq U\subset X$ is a vertical
$\mathbb{A}_{k}^{r}$-cylinder with respect to $f$ then 
\[
U\times_{Y}\mathrm{Spec}(K)\simeq(Z\times\mathbb{A}_{k}^{r})\times_{Y}\mathrm{Spec}(K)\simeq(Z\times_{\mathrm{Spec}(k)}\mathrm{Spec}(K))\times\mathbb{A}_{K}^{r}
\]
is an $\mathbb{A}_{K}^{r}$-cylinder contained in the fiber $X_{\eta}$
of $f$ over the generic point $\eta$ of $Y$. Conversely, if $V\simeq T\times\mathbb{A}_{K}^{r}$
is an $\mathbb{A}_{K}^{r}$-cylinder inside $X_{\eta}$, then let
$\Delta$ be the closure of $X_{\eta}\setminus V$ in $X$ and let
$X_{0}=X\setminus\Delta$. The projection $\mathrm{pr}_{T}:V\rightarrow T$
induces a rational map $\rho:X_{0}\dashrightarrow\overline{T}$ to
a projective model $\overline{T}$ of the closure of $T$ in $X$,
whose generic fiber is isomorphic to the affine space $\mathbb{A}_{K'}^{r}$
over the function field $K'$ of $T$. It follows that there exists
an open subset $\overline{T}_{0}$ of $\overline{T}$ over which $\rho$
is regular such that $\rho^{-1}(\overline{T}_{0})\simeq\overline{T}_{0}\times\mathbb{A}_{k}^{r}$.
Replacing $\overline{T}_{0}$ if necessary by a smaller open subset
$Z$ on which the rational map $\overline{T}\dashrightarrow Y$ induced
by $f$ is regular, we obtain an $\mathbb{A}_{k}^{r}$-cylinder $U=\rho^{-1}(Z)$
in $X$ which is vertical with respect to $f$. 
\end{proof}
Before we proceed to the proof of Theorem \ref{thm:dP-Main} in the
next three subsections, we list some corollaries of it.
\begin{cor}
Let $\pi:V\rightarrow C$ be a complex del Pezzo fibration over a
curve $C$. If $\deg(V/C)\geq5$ then $V$ contains a vertical $\mathbb{A}^{1}$-cylinder. \end{cor}
\begin{proof}
Indeed, the generic fiber of $\pi$ is then a smooth del Pezzo surface
over the function field $K$ of $C$, and hence it has a $K$-rational
point by virtue of the Tsen-Lang Theorem (see e.g. \cite[Theorem 3.12]{Has06}). \end{proof}
\begin{cor}
Every complex del Pezzo fibration $\pi:V\rightarrow W$ of degree
$\deg(V/W)=5$ contains a vertical $\mathbb{A}^{1}$-cylinder.\end{cor}
\begin{proof}
Indeed, by virtue of \cite{SD72}, every del Pezzo surface of degree
$5$ over a field $K$ has a $K$-rational point. 
\end{proof}

\subsection{Rationality of del Pezzo surfaces with Picard number one containing
a cylinder}

Our first observation is that the rationality of a del Pezzo surface
$S$ of Picard number one is a necessary condition for the existence
of an $\mathbb{A}^{1}$-cylinder in it. 
\begin{prop}
\label{cor:Rational-DP} Let $S$ be a smooth del Pezzo surface defined
over a field $K$ of characteristic zero $0$ with Picard number one.
If $S$ contains an $\mathbb{A}_{K}^{1}$ cylinder $U\simeq Z\times\mathbb{A}_{K}^{1}$
over a smooth curve $Z$ defined over $K$, then $S$ is rational.
\end{prop}
Given an $\mathbb{A}_{K}^{1}$-cylinder $U\simeq Z\times\mathbb{A}_{K}^{1}$
inside $S$, the projection $\mathrm{pr}_{Z}:U\rightarrow Z$ induces
a rational map $\varphi:S\dashrightarrow\overline{Z}$ over the smooth
projective model $\overline{Z}$ of $Z$. Note that $S$ is rational
provided that $\overline{Z}$ is rational, hence isomorphic to $\mathbb{P}_{K}^{1}$.
The assertion of the proposition is thus a direct consequence of the
following lemma:
\begin{lem}
\label{lem:Rational-Cylinder} Under the assumption of Proposition
\ref{cor:Rational-DP}, the following hold: 

a) The rational map $\varphi:S\dashrightarrow\overline{Z}$ is not
regular and it as a unique proper base point, which is $K$-rational. 

b) The curve $\overline{Z}$ is rational. \end{lem}
\begin{proof}
Since $\mathrm{Pic}(S)\simeq\mathbb{Z}$, $\varphi$ is strictly rational.
Indeed, otherwise it would be a $\mathbb{P}^{1}$-fibration admitting
a section $H$ defined over $K$ and the natural homomorphism $\mathrm{Pic}(\overline{Z})\oplus\mathbb{Z}\langle H\rangle\rightarrow\mathrm{Pic}(S)$
would be injective. The rational map $\varphi_{\overline{K}}:S_{\overline{K}}\dashrightarrow\overline{Z}_{\overline{K}}$
obtained by base change to the algebraic closure $\overline{K}$ of
$K$ is again strictly rational, extending the $\mathbb{A}^{1}$-fibration
$\mathrm{pr}_{Z_{\overline{K}}}:U_{\overline{K}}\simeq Z_{\overline{K}}\times\mathbb{A}_{\overline{K}}^{1}\rightarrow Z_{\overline{K}}$.
Since the closed fibers of $\mathrm{pr}_{Z_{\overline{K}}}$ are isomorphic
to $\mathbb{A}_{\overline{K}}^{1}$, a general member of $\varphi_{\overline{K}}$
is a projective curve with a unique place at infinity. It follows
that $\varphi_{\overline{K}}$ has a unique proper base point. This
implies in turn that $\varphi:S\dashrightarrow\overline{Z}$ has a
unique proper base point $p$, which is necessarily $K$-rational.
The same argument implies that a minimal resolution $\sigma:\tilde{S}\rightarrow S$
of the indeterminacies of $\varphi$ is obtained by blowing-up a finite
sequence of $K$-rational points, the last exceptional divisor produced
being a section $H\simeq\mathbb{P}_{K}^{1}$ of the resulting $\mathbb{P}^{1}$-fibration
$\tilde{\varphi}=\varphi\circ\sigma:\tilde{S}\rightarrow\overline{Z}$.
This implies that $\overline{Z}\simeq\mathbb{P}_{K}^{1}$. \end{proof}
\begin{rem}
In the previous lemma, the hypothesis that the Picard number of $S$
is equal to one is crucial to infer that the projection morphism $\mathrm{pr}_{Z}:U\simeq Z\times\mathbb{A}_{K}^{1}\rightarrow Z$
cannot extend to an everywhere defined $\mathbb{P}^{1}$-fibration
$\varphi:S\rightarrow\overline{Z}$. Indeed, for instance, letting
$\overline{Z}$ be a smooth geometrically rational curve without $K$-rational
point, $S=\overline{Z}\times\mathrm{Proj}_{K}(K[x,y])$ is a smooth
del Pezzo surface of Picard number $\rho_{K}(S)=2$ and degree $8$,
without $K$-rational point, hence in particular non $K$-rational,
containing a cylinder $U=S\setminus(\overline{Z}\times[0:1])\simeq\overline{Z}\times\mathbb{A}_{K}^{1}$.
\end{rem}

\subsection{\label{sub:Non-existence}Non-existence of $\mathbb{A}^{1}$-cylinders
in del Pezzo surfaces of degree $d\leq4$ }

Recall \cite[29.4.4., (iii), p.159]{Manin}  that the degree $d$
of a smooth del Pezzo surface $S$ of Picard number one ranges over
the set $\{1,2,3,4,5,6,8,9\}$. By classical results of Segre-Manin
\cite{Se43,Man66} and Iskovskikh \cite{Isk96}, such a surface of
degree $d\leq4$ is not rational over $K$. Combined with Proposition
\ref{cor:Rational-DP}, this implies the following:
\begin{prop}
\label{prop:noCylleq4} A smooth del Pezzo surface defined over a
field $K$ of characteristic zero $0$ of degree $d\leq4$ and with
Picard number one does not contain any $\mathbb{A}_{K}^{1}$-cylinder.\end{prop}
\begin{proof}
We find enlightening to give an alternative argument which does not
explicitly rely on the aforementioned results of Segre-Manin-Iskovskikh.
So suppose for contradiction that $S$ contains an open subset $U\simeq Z\times\mathbb{A}_{K}^{1}$
for a certain smooth curve $Z$ defined over $K$. By Lemma \ref{lem:Rational-Cylinder},
the rational map $\varphi:S\dashrightarrow\overline{Z}$ to the smooth
projective model $\overline{Z}\simeq\mathbb{P}_{K}^{1}$ of $Z$ induced
by the projection $\mathrm{pr}_{Z}:S\rightarrow Z$ has a unique proper
base point $p$, which is a $K$-rational point of $S$. Let $\sigma:\tilde{S}\rightarrow S$
be a minimal resolution of the indeterminacies of $\varphi$, consisting
of a finite sequence of blow-ups of $K$-rational points, with successive
exceptional divisors $E_{i}\simeq\mathbb{P}_{K}^{1}$, $i=1,\ldots,n$,
the last one being a section of the resulting $\mathbb{P}^{1}$-fibration
$\tilde{\varphi}=\varphi\circ\sigma:\tilde{S}\rightarrow\overline{Z}$,
and $\mathcal{L}$ be the mobile linear system on $S$ corresponding
to $\varphi:S\dashrightarrow\overline{Z}$. Since $\mathrm{Pic}(S)$
is generated by $-K_{S}$, it follows that $\mathcal{L}$ is contained
in the complete linear system $\left|-\mu K_{S}\right|$ for some
positive integer $\mu$. Letting $\tilde{\mathcal{L}}$ be the proper
transform of $\mathcal{L}$ on $\tilde{S}$, we have 
\[
K_{\tilde{S}}+\frac{1}{\mu}\tilde{\mathcal{L}}=\sigma^{*}(K_{S}+\frac{1}{\mu}\mathcal{L})+\sum_{i=1}^{n}\alpha_{i}E_{i}
\]
for some rational numbers $\alpha_{i}$. Since $E_{n}$ is a section
of $\tilde{\varphi}$ while the divisors $E_{i}$, $i=1,\ldots,n-1$
are contained in fibers of $\tilde{\varphi}$, taking the intersection
with a general $K$-rational fiber $F\simeq\mathbb{P}_{K}^{1}$ of
$\tilde{\varphi}$, we obtain: 
\[
-2=K_{\tilde{S}}\cdot F=(K_{\tilde{S}}+\frac{1}{\mu}\tilde{\mathcal{L}})\cdot F=\sigma^{*}(K_{S}+\frac{1}{\mu}\mathcal{L})\cdot F+\sum_{i=1}^{n}\alpha_{i}E_{i}\cdot F=\alpha_{n}
\]
as $K_{S}+\frac{1}{\mu}\mathcal{L}\sim_{\mathbb{Q}}0$ by the choice
of $\mu$. The pair $(S,\frac{1}{\mu}\mathcal{L})$ is thus not log-canonical
at $p$. By \cite[Theorem 3.1 p. 275]{CPR}, the local intersection
multiplicity $\mathcal{L}_{p}^{2}$ at $p$ of two members of $\mathcal{L}$
over general $K$-rational points of $\overline{Z}$ then satisfies
$\mathcal{L}_{p}^{2}>4\mu^{2}$. But on the other hand, since $p$
is the unique proper base point of $\varphi$, we have $\mathcal{L}_{p}^{2}=\mathcal{L}^{2}=(-\mu K_{S})^{2}=d\mu^{2}$,
a contradiction since $d\leq4$ by hypothesis. \end{proof}
\begin{rem}
The main ingredient of the proof of Proposition \ref{prop:noCylleq4} lies in an application of Corti's version of Pukhlikov's $4n^2$ inequality to the pair $(S,\frac{1}{\mu}\mathcal{L})$, which is not log-canonical at $p$. We could  have also infered the same result for the case $d \leqq 3$ from the fact that any smooth del Pezzo surface $X$ of degree $d \leqq 3$ defined over an algebraically closed field  of characteristic zero does not have an anti-canonically polar cylinder \cite{CPW,KPZ3}. Indeed, since the Picard group of $S$ is generated by $-K_S$, every $\mathbb{A}^1_K$-cylinder $U\subset S$ is automatically anti-canonically polar. The existence of such a cylinder would imply in turn that the base extension $X=S_{\overline{K}}$ of $S$ to the algebraic closure $\overline{K}$ of $K$ has an anti-canonically polar $\mathbb{A}_{\overline{K}}^1$-cylinder in contradiction to the aforementioned fact. Note that this argument is no longer applicable to the case $d=4$, as every smooth del Pezzo surface defined over an algebraically closed field of characteristic zero does contain anti-canonically polar cylinders as explained in Example \ref{Ex:AntiPol-cyl-4} below. \end{rem}
\begin{example}
\label{Ex:AntiPol-cyl-4}Every smooth del Pezzo surface $X$ of degree $4$ defined over an algebraically closed field of characteristic zero is isomorphic to the blow-up $h: S \rightarrow \mathbb{P}^2$ of $\mathbb{P}^2$ at five points in general position, say $p_1, \cdots , p_5$. Let $C$ be the smooth conic passing through all of $p_i$'s and let $\ell$ be the tangent line to $C$ at a point $q$ distinct  from the $p_i$'s. For a rational number $\epsilon >0$, we have \[ -K_S \sim_{\mathbb{Q}} (1+\epsilon) C' + (1-2 \epsilon)\ell' + \epsilon \sum_{i=1}^5 E_i, \] where $C'$ and $\ell'$ are the respective proper transforms of $C$ and $\ell$, and $E_i:=h^{-1}(p_i)$ ($1\leqq i \leqq 5)$. For sufficiently small $\epsilon$, the right hand side of the above expression is an effective divisor and on the other hand, we have: \[ S \backslash (C' \cup \ell' \cup E_1 \cup \cdots \cup E_5) \cong \mathbb{P}^2 \backslash (C \cup \ell) \cong {(\mathbb{A}_*^1)} \times \mathbb{A}^1. \] It follows that $S$ has a $(-K_S)$-polar $\mathbb{A}^1$-cylinder, actually a one-dimensional family of such cylinders parametrized by the points $q\in C\setminus\{p_1,\ldots,p_5\}$.  
\end{example}

\subsection{\label{sub:cylinders-exist} $\mathbb{A}^{1}$-cylinders in del Pezzo
surfaces of degree $d\geq5$}

By virtue of \cite[Theorem 29.4]{Manin}, a smooth del Pezzo surface
of degree $d\geq5$ and arbitrary Picard number is rational if and
only if it admits a $K$-rational point. On the other hand it follows
from Proposition \ref{cor:Rational-DP} that for such surfaces with
Picard number one, rationality is a necessary condition for the existence
of an $\mathbb{A}_{K}^{1}$-cylinder. The following result implies
in particular these two properties are in fact equivalent: 
\begin{prop}
\label{prop:CylGeq5} A smooth del Pezzo surface $S$ defined over
a field $K$ of characteristic zero $0$, of degree $d\geq5$ and
with a $K$-rational point admits an $\mathbb{A}_{K}^{1}$-cylinder.\end{prop}
\begin{proof}
We only treat the case where $S$ has Picard number one, the other
cases are similar and left to the reader. Building on the proof of
\cite[Theorem 29.4]{Manin}, we exhibit below an explicit $\mathbb{A}_{K}^{1}$-cylinder
in $S$ for each $d=9$, $8$, $6$ and $5$ respectively.

(i) If $d=9$ then by \cite[29.4.4 (i)]{Manin}, $S$ is isomorphic
over $K$ to $\mathbb{P}_{K}^{2}$ and we obtain an $\mathbb{A}_{K}^{1}$-cylinder
$U\simeq\mathbb{A}_{K}^{2}\simeq\mathbb{A}_{K}^{1}\times\mathbb{A}_{K}^{1}$
by taking the complement of a line $L\simeq\mathbb{P}_{K}^{1}$.

(ii) If $d=8$, then since $\rho(S)=1$, it follows from \cite[29.4.4 (ii)]{Manin}
that $S$ is isomorphic to a smooth quadric in $\mathbb{P}_{K}^{3}$
which is a nontrivial $K$-form of $\mathbb{P}_{K}^{1}\times\mathbb{P}_{K}^{1}$,
i.e., $S$ is not isomorphic to $\mathbb{P}_{K}^{1}\times\mathbb{P}_{K}^{1}$
but the base extension $S_{\overline{K}}$ of $S$ to the algebraic
closure $\overline{K}$ of $K$ is isomorphic to $\mathbb{P}_{\overline{K}}^{1}\times\mathbb{P}_{\overline{K}}^{1}$.
Given a $K$-rational point $p$ of $S$, the union in $S_{\overline{K}}\simeq\mathbb{P}_{\overline{K}}^{1}\times\mathbb{P}_{\overline{K}}^{1}$
of the fibers of the first and second projections passing through
$p$ is a curve $C$ defined over $K$. Letting $g:S'\rightarrow S$
be the blow-up of $S$ at $p$, with exceptional divisor $E\simeq\mathbb{P}_{K}^{1}$,
the proper transform of $C$ in $S'_{\overline{K}}$ is a pair of
disjoint $(-1)$-curves whose union is defined over $K$. The composition
$\varphi:S\dashrightarrow\mathbb{P}_{K}^{2}$ of $g^{-1}$ with the
contraction of the proper transform of $C$ is thus a rational map
defined over $K$, mapping $E$ to a line and restricting to an isomorphism
between $U=S\setminus C\simeq S'\setminus(C\cup E)$ and the complement
of the image of $E$ in $\mathbb{P}_{K}^{2}$, which isomorphic to
$\mathbb{A}_{K}^{2}$. 

(iii) If $d=6$, then $S_{\overline{K}}$ is isomorphic to the blow-up
$\tau:S_{\overline{K}}\rightarrow\mathbb{P}_{\overline{K}}^{2}$ of
$\mathbb{P}_{\overline{K}}^{2}$ at three points $x_{1}$, $x_{2}$,
$x_{3}$ in general position. The union $D$ of all the $(-1)$-curves
in $S_{\overline{K}}$ is defined over $K$, and consists of the exceptional
divisors of $\tau$ and the proper transforms of the lines in $\mathbb{P}_{\overline{K}}^{2}$
passing through $x_{i}$ and $x_{j}$, $1\leq i<j\leq3$. Since $\rho(S)=1$
it follows from \cite[29.4.4 (iv)]{Manin} that a $K$-rational point
$p$ of $S$ is necessarily supported outside of $D$. Let $g:S'\rightarrow S$
be the blow-up of a $K$-rational point $p$, with exceptional divisor
$E_{0}\simeq\mathbb{P}_{K}^{1}$. Then $S'$ is a del Pezzo surface
of degree $5$, and the union $D'$ of all $(-1)$-curves in $S'_{\overline{K}}$
consists of $E_{0,\overline{K}}$ , the proper transform of $D$ and
the proper transforms $\ell_{i}$ by $\tau\circ g_{\overline{K}}$
of the lines in $\mathbb{P}_{\overline{K}}^{2}$ passing through $\tau(p)$
and $x_{i}$, $i=1,\ldots,3$. By construction, $E_{0,\overline{K}}$
is invariant under the action of the Galois group $\mathrm{Gal}(\overline{K}/K)$
on $S'_{\overline{K}}$ and $\ell_{1}$, $\ell_{2}$, $\ell_{3}$
are the only $(-1)$-curves in $S'_{\overline{K}}$ intersecting $E_{0,\overline{K}}$.
The union $\ell_{1}\cup\ell_{2}\cup\ell_{3}$ is thus defined over
$K$, and since these curves are disjoint in $S'_{\overline{K}}$,
they can be simultaneously contracted, giving rise to a birational
morphism $h:S'\rightarrow S''$ defined over $K$ onto a smooth del
Pezzo surface $S''$ of Picard number one and degree $8$, hence again
a nontrivial $K$-form of $\mathbb{P}_{K}^{1}\times\mathbb{P}_{K}^{1}$,
containing at least a $K$-rational point $q$ supported on the image
$E_{0}''$ of $E_{0}$. The image of $E''_{0,\overline{K}}$ in $S''_{\overline{K}}$
is an irreducible curve of type $(1,1)$ in the divisor class group
$\mathrm{Cl}(S''_{\overline{K}})\simeq\mathbb{Z}^{2}$ of $S''_{\overline{K}}$.
The union of the fibers of the first and second projection passing
through $q$ in $S''_{\overline{K}}$ is another curve $C$ of type
$(1,1)$ defined over $K$. So $E''_{0,\overline{K}}$ and $C$ generate
a pencil $\varphi:S''\dashrightarrow\mathbb{P}_{K}^{1}$ defined over
$K$, having $q$ as a unique proper base point. The latter restricts
to a trivial $\mathbb{A}_{K}^{1}$-bundle $S''\setminus(E_{0}''\cup C)\rightarrow\mathbb{P}_{K}^{1}\setminus(\varphi(E_{0}'')\cup\varphi(C))\simeq\mathbb{A}_{*,K}^{1}$
over $\mathbb{A}_{*,K}^{1}\simeq\mathrm{Spec}(K[t^{\pm1}])$. By construction,
$g\circ h^{-1}$ induces an isomorphism between $S''\setminus(E_{0}''\cup C)$
and its image $U\subset S$, and $\varphi\circ h\circ g^{-1}:U\rightarrow\mathbb{A}_{*,K}^{1}$
is isomorphic to trivial $\mathbb{A}_{K}^{1}$-bundle $\mathbb{A}_{*,K}^{1}\times\mathbb{A}_{K}^{1}$
over $\mathbb{A}_{*,K}^{1}$. 

(iv) If $d=5$ then $S_{\overline{K}}$ is isomorphic to the blow-up
$\tau:S_{\overline{K}}\rightarrow\mathbb{P}_{\overline{K}}^{2}$ of
$\mathbb{P}_{\overline{K}}^{2}$ at four points $x_{1},\ldots,x_{4}$
in general position. The union $D$ of all the $(-1)$-curves in $S_{\overline{K}}$
is defined over $K$, and consists of the exceptional divisors of
$\tau$ and the proper transforms of the lines in $\mathbb{P}_{\overline{K}}^{2}$
passing through $x_{i}$ and $x_{j}$, $1\leq i<j\leq4$. Since $\rho(S)=1$
it follows from \cite[29.4.4 (v)]{Manin} that a $K$-rational point
$p$ of $S$ is necessarily supported outside of $D$. Let $g:S'\rightarrow S$
be the blow-up of such a $K$-rational point $p$, with exceptional
divisor $E_{0}\simeq\mathbb{P}_{K}^{1}$. Then $S'$ is a del Pezzo
surface of degree $4$, and the union $D'$ of all $(-1)$-curves
in $S'_{\overline{K}}$ consists of $E_{0,\overline{K}}$ , the proper
transform of $D$, the proper transforms $\ell_{i}$ by $\tau\circ g_{\overline{K}}$
of the lines in $\mathbb{P}_{\overline{K}}^{2}$ passing through $\tau(p)$
and $x_{1},\ldots,x_{4}$, and the proper transform $C$ of the unique
smooth conic in $\mathbb{P}_{\overline{K}}^{2}$ passing through $\tau(p)$,
$x_{1},\ldots,x_{4}$. Since $C$, $\ell_{1},\ldots,\ell_{4}$ are
the only $(-1)$-curves intersecting the proper transform of $E_{0,\overline{K}}$,
their union is defined over $K$, and since they are also disjoint,
they can therefore be simultaneously contracted. This yields a birational
morphism $h:S'\rightarrow S''$ defined over $K$ onto a smooth del
Pezzo surface of degree $9$, containing a $K$-rational point supported
on the image $E_{0}''$ of $E_{0}$. So $S''\simeq\mathbb{P}_{K}^{2}$
in which $E_{0}''$ is a smooth conic, with a $K$-rational point
$q$. Then $E_{0}''$ and twice is tangent line $T_{q}(E_{0}'')\simeq\mathbb{P}_{K}^{1}$
at $p$ generate a pencil $\varphi:S''\dashrightarrow\mathbb{P}_{K}^{1}$
defined over $K$, having $q$ as a unique proper base point and whose
restriction to $\mathbb{P}_{K}^{2}\setminus(E_{0}''\cup T_{q}(E_{0}''))$
is a trivial $\mathbb{A}_{K}^{1}$-bundle over $\mathbb{P}_{K}^{1}\setminus(\varphi(E_{0}'')\cup\varphi(T_{q}(E_{0}'')))\simeq\mathbb{A}_{*,K}^{1}$.
By construction, $g\circ h^{-1}$ induces an isomorphism between $S''\setminus(E_{0}''\cup T_{q}(E_{0}''))$
and its image $U\subset S$, and $\varphi\circ h\circ g^{-1}:U\rightarrow\mathbb{A}_{*,K}^{1}$
is isomorphic to trivial $\mathbb{A}_{K}^{1}$-bundle $\mathbb{A}_{*,K}^{1}\times\mathbb{A}_{K}^{1}$
over $\mathbb{A}_{*,K}^{1}$. 
\end{proof}

\subsection{\label{sub:A2-cyl} Complement: $\mathbb{A}^{2}$-cylinders in del
Pezzo surfaces of Picard number one}

It follows from the proof of Proposition \ref{prop:CylGeq5} that
every smooth del Pezzo surface $S$ of degree $d\geq8$ with a $K$-rational
point actually contains $\mathbb{A}_{K}^{2}$ as an open subset. Here
we show in contrast that this is no longer the case for del Pezzo
surfaces of degree $5$ and $6$ with Picard number one, namely:
\begin{prop}
A smooth del Pezzo surface $S$ defined over a field $K$ of characteristic
zero $0$, with Picard number one and of degree $d=5$ or $6$ does
not contain $\mathbb{A}_{K}^{2}$ as an open subset. \end{prop}
\begin{proof}
We proceed by contradiction, assuming that $S$ contains an open subset
$U$ isomorphic to $\mathbb{A}_{K}^{2}$. Since $\mathrm{Pic}(S)$
is generated by $-K_{S}$ and $\mathrm{Pic}(U)$ is trivial, $B=S\setminus U$
is a reduced, irreducible effective anti-canonical divisor on $S$,
defined over $K$. Furthermore, the irreducible components of $B_{\overline{K}}$
must form a basis of the Picard group of $S_{\overline{K}}$ on which
the Galois group $\mathrm{Gal}(\overline{K}/K)$ acts transitively. 

(i) If $d=6$, then $\mathrm{Pic}(S_{\overline{K}})\simeq\mathbb{Z}^{4}$
and we can write $B_{\overline{K}}=B_{1}+B_{2}+B_{3}+B_{4}$, where
$B_{1},\ldots,B_{4}$ are irreducible curves forming a basis of $\mathrm{Pic}(S_{\overline{K}})$
on which $\mathrm{Gal}(\overline{K}/K)$ acts transitively. The equality
\[
6=(-K_{S_{\overline{K}}}^{2})=B_{\overline{K}}^{2}=-\sum_{i=1}^{4}K_{S_{\overline{K}}}\cdot B_{i}
\]
implies that $K_{S_{\overline{K}}}\cdot B_{i}=-1$ for at least one
$i$. The corresponding curve is thus a $(-1)$-curve, and so, the
$B_{i}$ are all $(-1)$-curves. This would imply in turn that $4=6$,
which is absurd. 

(ii) If $d=5$, then $\mathrm{Pic}(S_{\overline{K}})\simeq\mathbb{Z}^{5}$
and similarly as above, we can write $B_{\overline{K}}=\sum_{i=1}^{5}B_{i}$,
where $B_{1},\ldots,B_{5}$ are irreducible curves forming a basis
of $\mathrm{Pic}(S_{\overline{K}})$ on which $\mathrm{Gal}(\overline{K}/K)$
acts transitively. The equality 
\[
5=(-K_{S_{\overline{K}}}^{2})=B_{\overline{K}}^{2}=-\sum_{i=1}^{5}K_{S_{\overline{K}}}\cdot B_{i}
\]
implies again that all the $B_{i}$ are $(-1)$-curves. Since $S_{\overline{K}}\setminus B_{\overline{K}}\simeq\mathbb{A}_{\overline{K}}^{2}$,
it follows from \cite{Morrow} that the support of the total transform
of $B_{\overline{K}}$ in a minimal log-resolution of the pair $(S_{\overline{K}},B_{\overline{K}})$
is a tree of rational curves. Therefore, the support of $B_{\overline{K}}$
is connected and does contain any cycle, and since $\mathrm{Gal}(\overline{K}/K)$
acts transitively on its irreducible components, we conclude that
the curves $B_{i}$ intersect each others in a unique common point.
Letting $c=\min_{1\leq i<j\leq5}\{(B_{i}\cdot B_{j})\}\geq1$ be the
minimum of the intersection number between two distinct components
of $B_{\overline{K}}$, we obtain that 
\[
5=(-K_{S_{\overline{K}}}^{2})=(B_{1}+\cdots+B_{5})^{2}\geq-5+20c,
\]
which is absurd as $c\geq1$. \end{proof}
\begin{cor}
A complex del Pezzo fibration $\pi:V\rightarrow W$ admits a vertical
$\mathbb{A}^{2}$-cylinder if and only if $\deg(V/W)=8$ or $9$ and
$\pi:V\rightarrow W$ has a rational section. 
\end{cor}

\section{Examples of threefold del Pezzo fibrations containing twisted cylinders}

In what follows we first briefly review the general setup for the
construction of projective completions of $\mathbb{A}_{\mathbb{C}}^{3}$
into total spaces of del Pezzo fibrations $\pi:V\rightarrow\mathbb{P}^{1}$
established in \cite{DK3}. Then we give a detailed construction for
the specific case $d=4$, which was announced without proof in \cite{DK3},
thus completing the proof of Theorem \ref{thm:MainThm2}.

\subsection{General setup and existence results in degree $\leq3$ }

\subsubsection{Step 1: Pencils of del Pezzo surfaces }

We begin with a smooth del Pezzo surface $S$ of degree $d\leq3$
anti-canonically embedded as a hypersurface of degree $e$ in a weighted
projective space $\mathbb{P}=\mathrm{Proj}(\mathbb{C}[x,y,z,w])$.
So $\mathbb{P}$ is equal $\mathbb{P}^{3}$, $\mathbb{P}(1,1,1,2)$
and $\mathbb{P}(1,1,2,3)$ and $e$ is equal to $3$, $4$ or $6$
according as $d=3$, $2$ and $1$. Given a hyperplane $H\in|\mathcal{O}_{\mathbb{P}}(1)|$,
the open subset $U=\mathbb{P}\setminus H$ is isomorphic to $\mathbb{A}_{\mathbb{C}}^{3}$.
We let $\mathcal{L}\subset|\mathcal{O}_{\mathbb{P}}(e)|$ the pencil
generated by $S$ and $eH$ and we denote by $\overline{f}:\mathbb{P}\dashrightarrow\mathbb{P}^{1}$
the corresponding rational map. We let $\infty=\overline{f}_{*}(H)\in\mathbb{P}^{1}$.

\subsubsection{Step 2: Good resolutions }

Next we take a \emph{good resolution} of the indeterminacies of $\overline{f}$,
that is, a triple $(\tilde{\mathbb{P}},\sigma,\tilde{f})$ consisting
of a projective threefold $\tilde{\mathbb{P}}$, a birational morphism
$\sigma:\tilde{\mathbb{P}}\rightarrow\mathbb{P}$ and a morphism $\tilde{f}:\tilde{\mathbb{P}}\rightarrow\mathbb{P}^{1}$
satisfying the following properties:

a) The diagram \[\xymatrix{ \tilde{\mathbb{P}} \ar[r]^{\sigma} \ar[d]_{\tilde{f}} & \mathbb{P} \ar@{-->}[d]^{\overline{f}} \\ \mathbb{P}^1 \ar@{=}[r] & \mathbb{P}^1 }\]
commutes. 

b) $\tilde{\mathbb{P}}$ has at most $\mathbb{Q}$-factorial terminal
singularities and is smooth outside $\tilde{f}^{-1}(\infty)$. 

c) $\sigma:\tilde{\mathbb{P}}\rightarrow\mathbb{P}$ is a sequence
of blow-ups whose successive centers lie above the base locus of $\mathcal{L}$,
inducing an isomorphism $\tilde{\mathbb{P}}\setminus\sigma^{-1}(H)\stackrel{\sim}{\rightarrow}\mathbb{P}\setminus H$,
and whose restriction to every closed fiber of $\tilde{f}$ except
$\tilde{f}^{-1}(\infty)$ is an isomorphism onto its image. 

Such a good resolution $(\tilde{\mathbb{P}},\sigma,\tilde{f})$ of
$f:\mathbb{P}\dashrightarrow\mathbb{P}^{1}$ always exists and can
be for instance obtained by first taking the blow-up $\tau:X\rightarrow\mathbb{P}$
of the scheme-theoretic base locus of $\mathcal{L}$ and then any
resolution $\tau_{1}:\tilde{\mathbb{P}}\rightarrow X$ of the singularities
of $X$. In this case, the triple $(\tilde{\mathbb{P}},\tau\circ\tau_{1},\overline{f}\circ\pi\circ\tau_{1})$
is a good resolution of $\overline{f}$ for which $\tilde{\mathbb{P}}$
is even smooth.

The definition implies that the generic fiber $\tilde{\mathbb{P}}_{\eta}$
of $\tilde{f}$ is a smooth del Pezzo surface of degree $d$ defined
over the field of rational functions $K$ of $\mathbb{P}^{1}$. The
irreducible divisors in the exceptional locus $\mathrm{Exc}(\sigma)$
of $\sigma$ that are vertical for $\tilde{f}$, i.e. contained in
closed fibers of $\tilde{f}$, are all contained in $\tilde{f}^{-1}(\infty)$.
On the other hand, $\mathrm{Exc}(\sigma)$ contains exactly as many
irreducible horizontal divisors as there are irreducible components
in $H\cap S$, and $\sigma^{-1}(H)$ intersects $\tilde{\mathbb{P}}_{\eta}$
along the curve $D_{\eta}\simeq(H\cap S)\times_{\mathrm{Spec}(\mathbb{C})}\mathrm{Spec}(K)$
which is an anti-canonical divisor on $\tilde{\mathbb{P}}_{\eta}$
with the same number of irreducible components as $H\cap S$. Note
also that by assumption $\tilde{U}=\tilde{\mathbb{P}}\setminus\sigma^{-1}(H)$
is again isomorphic to $\mathbb{A}_{\mathbb{C}}^{3}$.

\subsubsection{Step 3: Relative MMP }

The next step consists in running a MMP $\varphi:\tilde{\mathbb{P}}_{0}=\tilde{\mathbb{P}}\dashrightarrow\tilde{\mathbb{P}}'=\tilde{\mathbb{P}}_{n}$
relative to the morphism $\tilde{f}_{0}=\tilde{f}:\tilde{\mathbb{P}}_{0}\rightarrow\mathbb{P}^{1}$.
Recall \cite[3.31]{KM98} that such a relative MMP consists of a finite
sequence $\varphi=\varphi_{n}\circ\cdots\circ\varphi_{1}$ of birational
maps 
\begin{eqnarray*}
\tilde{\mathbb{P}}_{k-1} & \stackrel{\varphi_{k}}{\dashrightarrow} & \tilde{\mathbb{P}}_{k}\\
\tilde{f}_{k-1}\downarrow &  & \downarrow\tilde{f}_{k}\qquad k=1,\ldots,n,\\
\mathbb{P}^{1} & = & \mathbb{P}^{1}
\end{eqnarray*}
where each $\varphi_{k}$ is associated to an extremal ray $R_{k-1}$
of the closure $\overline{NE}(\tilde{\mathbb{P}}_{k-1}/\mathbb{P}^{1})$
of the relative cone of curves of $\tilde{\mathbb{P}}_{k-1}$ over
$\mathbb{P}^{1}$. Each of these birational maps $\varphi_{k}$ is
either a divisorial contraction or a flip whose flipping and flipped
curves are contained in the fibers of $\tilde{f}_{k-1}$ and $\tilde{f}_{k}$
respectively. The following crucial result established in \cite{DK3}
guarantees that every such relative MMP terminates with a projective
threefold $\tilde{\mathbb{P}}'$ with at most $\mathbb{Q}$-factorial
terminal singularities containing $\mathbb{A}_{\mathbb{C}}^{3}$ as
an open subset. 
\begin{prop}
\label{prop:MMP-preserving} \label{lem:MMP-intermediate-steps} Let
$\mathcal{L}\subset|\mathcal{O}_{\mathbb{P}}(e)|$ be as above and
let $(\tilde{\mathbb{P}},\sigma,\tilde{f})$ be any good resolution
of the corresponding rational map $\overline{f}:\mathbb{P}\dashrightarrow\mathbb{P}^{1}$.
Then every MMP $\varphi:\tilde{\mathbb{P}}\dashrightarrow\tilde{\mathbb{P}}'$
relative to $\tilde{f}:\tilde{\mathbb{P}}\rightarrow\mathbb{P}^{1}$
restricts to an isomorphism $\mathbb{A}_{\mathbb{C}}^{3}\simeq\tilde{\mathbb{P}}\setminus\sigma^{-1}(H)\stackrel{\sim}{\rightarrow}\tilde{\mathbb{P}}'\setminus\varphi_{*}(\sigma^{-1}(H))$. 
\end{prop}
In particular, the restriction of every MMP $\varphi:\tilde{\mathbb{P}}\dashrightarrow\tilde{\mathbb{P}}'$
relative to $\tilde{f}:\tilde{\mathbb{P}}\rightarrow\mathbb{P}^{1}$
between the generic fibers $\tilde{\mathbb{P}}_{\eta}$ and $\tilde{\mathbb{P}}_{\eta}'$
of $\tilde{f}:\tilde{\mathbb{P}}\rightarrow\mathbb{P}^{1}$ and $\tilde{f}':\tilde{\mathbb{P}}\rightarrow\mathbb{P}^{1}$
respectively is either an isomorphism onto its image, or the contraction
of a finite sequence of successive $(-1)$-curves among the irreducible
components of $D_{\eta}=\tilde{\mathbb{P}}_{\eta}\cap\sigma^{-1}(H)$.
It was shown in addition in \cite{DK3} that for every $k=1,\ldots,n$,
the restriction of $\varphi_{k}$ to every closed fiber of $\tilde{f}_{k-1}:\tilde{\mathbb{P}}_{k-1}\rightarrow\mathbb{P}^{1}$
distinct from $\tilde{f}_{k-1}^{-1}(\infty)$ is either an isomorphism
onto the corresponding fiber of $\tilde{f}_{k}:\tilde{\mathbb{P}}_{k}\rightarrow\mathbb{P}^{1}$
or the contraction of finitely many disjoint $(-1)$-curves. In particular,
in the case where $\varphi_{k}:\tilde{\mathbb{P}}_{k-1}\dashrightarrow\tilde{\mathbb{P}}_{k}$
is a flip, then all its flipping and flipped curves are contained
in $\tilde{f}_{k-1}^{-1}(\infty)$ and $\tilde{f}_{k}^{-1}(\infty)$
respectively.

\subsubsection{Step 4: Determination of the possible outputs}

Since a general member of a pencil $\mathcal{L}\subset|\mathcal{O}_{\mathbb{P}}(e)|$
is a rational surface, the output $\tilde{\mathbb{P}}'$ of a relative
MMP $\varphi:\tilde{\mathbb{P}}\dashrightarrow\tilde{\mathbb{P}}'$
ran from a good resolution $(\tilde{\mathbb{P}},\sigma,\tilde{f})$
of the corresponding rational map $\overline{f}:\mathbb{P}\dashrightarrow\mathbb{P}^{1}$
is a Mori fiber space $\tilde{f}':\tilde{\mathbb{P}}'\rightarrow\mathbb{P}^{1}$.
So $\tilde{f}':\tilde{\mathbb{P}}'\rightarrow\mathbb{P}^{1}$ is either
a del Pezzo fibration with relative Picard number one, whose degree
is fully determined by that of the initial del Pezzo surface $S\subset\mathbb{P}$
and the number of $(-1)$-curves in the generic fiber of $\tilde{f}:\tilde{\mathbb{P}}\rightarrow\mathbb{P}^{1}$
contracted by $\varphi$, or it factors through a Mori conic bundle
$\xi:\tilde{\mathbb{P}}'\rightarrow X$ over a certain normal projective
surface $q:X\rightarrow\mathbb{P}^{1}$. The following theorem established
in \cite{DK3} shows that except maybe in the case where $d=3$ and
$H\cap S$ consists of two irreducible components, the structure of
$\tilde{\mathbb{P}}'$ depends only on the base locus of $\mathcal{L}$.
In particular, it depends neither on the chosen good resolution $(\tilde{\mathbb{P}},\sigma,\tilde{f})$
nor on a particular choice of a relative MMP $\varphi:\tilde{\mathbb{P}}\dashrightarrow\tilde{\mathbb{P}}'$. 
\begin{thm}
\label{thm:MMP-outputs} Let \textup{$\mathcal{L}\subset|\mathcal{O}_{\mathbb{P}}(e)|$}
be the pencil generated by a smooth del Pezzo surface $S\subset\mathbb{P}$
of degree $d\in\{1,2,3\}$ and $eH$ for some $H\in|\mathcal{O}_{\mathbb{P}}(1)|$,
let $(\tilde{\mathbb{P}},\sigma,\tilde{f})$ be a good resolution
of the corresponding rational map $\overline{f}:\mathbb{P}\dashrightarrow\mathbb{P}^{1}$,
and let $\varphi:\tilde{\mathbb{P}}\dashrightarrow\tilde{\mathbb{P}}'$
be a relative MMP. Then the following hold:

a) If $H\cap S$ is irreducible, then $\tilde{f}':\tilde{\mathbb{P}}'\rightarrow\mathbb{P}^{1}$
is a del Pezzo fibration of degree $d$. 

b) If $d=2$ and $H\cap S$ is reducible, then $\tilde{f}':\tilde{\mathbb{P}}'\rightarrow\mathbb{P}^{1}$
is del Pezzo fibration of degree $d+1=3$. 

c) If $H\cap S$ has three irreducible components, then $\tilde{f}':\tilde{\mathbb{P}}'\rightarrow\mathbb{P}^{1}$
factors through a Mori conic bundle $\xi:\tilde{\mathbb{P}}'\rightarrow X$
over a normal projective surface $q:X\rightarrow\mathbb{P}^{1}$. 
\end{thm}
\noindent As a consequence, we obtain the following existence result:
\begin{cor}
Let \textup{$\mathcal{L}\subset|\mathcal{O}_{\mathbb{P}}(e)|$} be
the pencil generated by a smooth del Pezzo surface $S\subset\mathbb{P}$
of degree $d\in\{1,2,3\}$ and $eH$ for some $H\in|\mathcal{O}_{\mathbb{P}}(1)|$
such that $H\cap S$ is irreducible. Then for every good resolution
$(\tilde{\mathbb{P}},\sigma,\tilde{f})$ of the corresponding rational
map $\overline{f}:\mathbb{P}\dashrightarrow\mathbb{P}^{1}$ and every
MMP $\varphi:\tilde{\mathbb{P}}\dashrightarrow\tilde{\mathbb{P}}'$
relative to $\tilde{f}:\tilde{\mathbb{P}}\rightarrow\mathbb{P}^{1}$,
the output $\tilde{f}':\tilde{\mathbb{P}}'\rightarrow\mathbb{P}^{1}$
is a del Pezzo fibration of degree $d$ whose total space $\tilde{\mathbb{P}}'$
contains $\mathbb{A}_{\mathbb{C}}^{3}$ as a Zariski open subset. 
\end{cor}

\subsection{Existence results in degree $4$}

In Theorem \ref{thm:MMP-outputs}, the remaining case where $S$ is
smooth cubic in $\mathbb{P}^{3}$ and $H\cap S$ consists of two irreducible
components, namely a line $L$ and smooth conic $C$ intersecting
each others twice, is more complicated. Here given a good resolution $(\tilde{\mathbb{P}},\sigma,\tilde{f})$ of the rational map $\overline{f}:\mathbb{P}=\mathbb{P}^{3}\dashrightarrow\mathbb{P}^{1}$, the intersection of $\sigma^{-1}(H)$ with the generic fiber $\tilde{\mathbb{P}}_{\eta}$ of $\tilde{f}:\tilde{\mathbb{P}}\rightarrow\mathbb{P}^{1}$ is a reduced anti-canonical divisor whose support consists of the union of a $(-1)$-curve $L_{\eta}$ and of a $0$-curve $C_{\eta}$ both defined over the function field $K$ of $\mathbb{P}^1$.
By Proposition \ref{prop:MMP-preserving}, the only horizontal divisors
contracted by a relative MMP $\varphi:\tilde{\mathbb{P}}\dashrightarrow\tilde{\mathbb{P}}'$
are irreducible components of $\sigma^{-1}(H)$. In this case, it
follows that $\varphi$ can contract at most the irreducible component
of $\sigma^{-1}(H)$ intersecting $\tilde{\mathbb{P}}_{\eta}$ along
$L_{\eta}$. Indeed, if $L_{\eta}$ is contracted at the certain step
$\varphi_{k}:\tilde{\mathbb{P}}_{k-1}\dashrightarrow\tilde{\mathbb{P}}_{k}$
then the image of $C_{\eta}$ in the generic fiber of $\tilde{f}_{k}:\tilde{\mathbb{P}}_{k}\rightarrow\mathbb{P}^{1}$
is a singular curve with positive self-intersection which therefore
cannot be contracted at any further step $\varphi_{k'}$, $k'\geq k+1$,
of $\varphi$. If $L_{\eta}$ is contracted, then the generic fiber
of the output $\tilde{f}':\tilde{\mathbb{P}}'\rightarrow\mathbb{P}^{1}$
of $\varphi$ is a smooth del Pezzo surface of degree $4$, and it
was established in \cite[Proposition 11]{DK3} that in this case,
$\tilde{f}':\tilde{\mathbb{P}}'\rightarrow\mathbb{P}^{1}$ is in fact
a del Pezzo fibration of degree $4$. The following result was announced
without proof in \cite{DK3}:
\begin{prop}
\label{prop:dp4-cylinder} Let $S\subset\mathbb{P}^{3}$ be a smooth
cubic surface, let $H\in\left|\mathcal{O}_{\mathbb{P}^{3}}(1)\right|$
be a hyperplane intersecting $S$ along the union of a line and smooth
conic, let $\mathcal{L}\subset\left|\mathcal{O}_{\mathbb{P}^{3}}(3)\right|$
be the pencil generated by $S$ and $3H$ and let $\overline{f}:\mathbb{P}^{3}\dashrightarrow\mathbb{P}^{1}$
be the corresponding rational map. Then there exists a good resolution
$(\tilde{\mathbb{P}},\sigma,\tilde{f})$ and a MMP $\varphi:\tilde{\mathbb{P}}\dashrightarrow\tilde{\mathbb{P}}'$
relative to $\tilde{f}:\tilde{\mathbb{P}}\rightarrow\mathbb{P}^{1}$
whose output is a del Pezzo fibration $\tilde{f}':\tilde{\mathbb{P}}'\rightarrow\mathbb{P}^{1}$
of degree $4$.

In particular, $\tilde{f}':\tilde{\mathbb{P}}'\rightarrow\mathbb{P}^{1}$
is a del Pezzo fibration of degree $4$ whose total space $\tilde{\mathbb{P}}'$
contains $\mathbb{A}_{\mathbb{C}}^{3}$ as a Zariski open subset. 
\end{prop}
\indent The rest of this subsection is devoted to the proof of this
proposition. In view of the above discussion, it is enough to construct
a particular good resolution $\sigma:\tilde{\mathbb{P}}\rightarrow\mathbb{P}^{3}$
of $\overline{f}:\mathbb{P}^{3}\dashrightarrow\mathbb{P}^{1}$ for
which there exists a MMP $\varphi:\tilde{\mathbb{P}}\dashrightarrow\tilde{\mathbb{P}}'$
relative to $\tilde{f}:\tilde{\mathbb{P}}\rightarrow\mathbb{P}^{1}$
whose first step consists of the contraction of the irreducible component
of $\sigma^{-1}(H)$ intersecting the generic fiber of $\tilde{f}$
along the $(-1)$-curve $L_{\eta}$.

\subsubsection{Construction of a particular good resolution $\sigma:\tilde{\mathbb{P}}\rightarrow\mathbb{P}^{3}$ }

Let again $\mathcal{L}\subset\left|\mathcal{O}_{\mathbb{P}^{3}}(3)\right|$
be the pencil generated by a smooth cubic surface $S\subset\mathbb{P}^{3}$
and $3H$, where $H\in\left|\mathcal{O}_{\mathbb{P}^{3}}(1)\right|$
is a hyperplane intersecting $S$ along the union of a line $L$ and
smooth conic $C$. In what follows, starting from $\mathbb{P}_{0}=\mathbb{P}^{3}$,
we construct a resolution of the base locus of $\mathcal{L}$ consisting
of a sequence of blow-ups 
\[
\sigma=\sigma_{1}\circ\cdots\circ\sigma_{6}:\tilde{\mathbb{P}}=\mathbb{P}_{6}\stackrel{\sigma_{6}}{\longrightarrow}\mathbb{P}_{5}\stackrel{\sigma_{5}}{\longrightarrow}\cdots\stackrel{\sigma_{2}}{\longrightarrow}\mathbb{P}_{1}\stackrel{\sigma_{1}}{\longrightarrow}\mathbb{P}_{0}=\mathbb{P}^{3}
\]
along successive smooth centers. To fix the notation, we let $S_i$ and $H_i$ the proper transforms of $S$ and $H$ on $\mathbb{P}_i$, respectively. Similarly, we denote by $\mathcal{L}_i$ the proper transform on $\mathbb{P}_i$ of the pencil $\mathcal{L}$. We denote by $E_i$  the exceptional divisor of $\sigma_i$, and we use the same notation to denote its proper transform on any other $\mathbb{P}_j$.
The base locus $\mathrm{Bs}\mathcal{L}$ of $\mathcal{L}$ is supported
by the union of $C$ and $L$. We proceed in two steps:

Step 1) First we let $\sigma_{1}:\mathbb{P}_{1}\rightarrow\mathbb{P}_{0}$
be the blow-up of $\mathbb{P}_{0}$ with center at $C$. Since $C$
is a Cartier divisor on $S$, $\sigma_{1}$ restricts to an isomorphism
between $S_{1}$ and $S$, in particular $S_{1}$ is again smooth
and $E_{1}\cap S_{1}$ is a smooth conic $C_{1}$ which is mapped
isomorphically onto $C$ by the restriction of $\sigma_{1}$. We have
$S_{1}\sim3H_{1}+2E_{1}$, $\mathcal{L}_{1}$ is spanned by $S_{1}$
and $3H_{1}+2E_{1}$, and $\mathrm{Bs}\mathcal{L}_{1}$ is supported
by the union of $C_{1}$ with the proper transform $L_{1}$ of $L$. 

Next we let $\sigma_{2}:\mathbb{P}_{2}\rightarrow\mathbb{P}_{1}$
be the blow-up of $\mathbb{P}_{1}$ with center at $C_{1}$. Similarly
as in the previous case, $\sigma_{2}$ restricts to an isomorphism
between $S_{2}$ and $S_{1}$, $E_{2}\cap S_{2}$ is a smooth conic
$C_{2}$ which is mapped isomorphically onto $C_{1}$ by $\sigma_{2}$.
Furthermore, $S_{2}\sim3H_{2}+2E_{1}+E_{2}$, the pencil $\mathcal{L}_{2}$
is spanned by $S_{2}$ and $3H_{2}+2E_{1}+E_{2}$, and its base locus
is supported on the union of $C_{2}$ with the proper transform $L_{2}$
of $L_{1}$. 

Then we let $\sigma_{3}:\mathbb{P}_{3}\rightarrow\mathbb{P}_{2}$
be the blow-up of $\mathbb{P}_{2}$ with center at $C_{2}$. Again,
$\sigma_{3}$ restricts to an isomorphism between $S_{3}$ and $S_{2}$.
We have $S_{3}\sim3H_{3}+2E_{1}+E_{2}$, the pencil $\mathcal{L}_{3}$
is spanned by $S_{3}$ and $3H_{3}+2E_{1}+E_{2}$, and $\mathrm{Bs}\mathcal{L}_{3}$
is supported by the line $L_{3}$ which is the proper transform of
$L$ by the isomorphism $S_{3}\stackrel{\sim}{\rightarrow}S$ induced
by $\sigma_{1}\circ\sigma_{2}\circ\sigma_{3}$. 

Step 2). We let $\sigma_{4}:\mathbb{P}_{4}\rightarrow\mathbb{P}_{3}$
be the blow-up of $\mathbb{P}_{3}$ with center at $L_{3}$. A similar
argument as above implies that $S_{4}\sim3H_{4}+2E_{1}+E_{2}+2E_{4}$,
that $\mathcal{L}_{4}$ is spanned by $S_{4}$ and $3H_{4}+2E_{1}+E_{2}+2E_{4}$
and that its base locus is supported on the line $L_{4}=E_{4}\cap S_{4}$
which is mapped isomorphically onto $L_{3}$ by the restriction of
$\sigma_{4}$ to $S_{4}$.

Then we let $\sigma_{5}:\mathbb{P}_{5}\rightarrow\mathbb{P}_{4}$
be the blow-up of $\mathbb{P}_{4}$ with center at $L_{4}$. The pencil
$\mathcal{L}_{5}$ is generated by $S_{5}$ and $3H_{5}+2E_{1}+E_{2}+2E_{4}+E_{5}$,
and its base locus is equal to the line $L_{5}=E_{5}\cap S_{5}$. 

Finally, we let $\sigma_{6}:\mathbb{P}_{6}\rightarrow\mathbb{P}_{5}$
be the blow-up of $\mathbb{P}_{5}$ with center at $L_{5}$. By construction,
$\mathbb{P}_{6}$ is a smooth threefold on which the proper transform
$\mathcal{L}_{6}$ of $\mathcal{L}$ is the base point free pencil
generated by $S_{6}\simeq S$ and $3H_{6}+2E_{1}+E_{2}+2E_{4}+E_{5}$. 

Summing up, we obtained:
\begin{lem}
With the notation above, the following hold:

a) The birational morphism $\sigma=\sigma_{1}\circ\cdots\cdots\sigma_{6}:\tilde{\mathbb{P}}=\mathbb{P}_{6}\rightarrow\mathbb{P}^{3}$
is a good resolution of $\overline{f}:\mathbb{P}^{3}\dashrightarrow\mathbb{P}^{1}$. 

b) The divisors $E_{3}$ and $E_{6}$ are the horizontal irreducible
components of $\sigma^{-1}(H)$, and they intersect the generic fiber
of $\tilde{f}:\tilde{\mathbb{P}}\rightarrow\mathbb{P}^{1}$ along
a $0$-curve $C_{\eta}$ and $(-1)$-curve $L_{\eta}$ respectively. 
\end{lem}

\subsubsection{Existence of a suitable relative MMP }

To complete the proof of Proposition \ref{prop:dp4-cylinder}, it
remains to check that $E_{6}$ can be contracted at a step of a MMP
$\varphi:\tilde{\mathbb{P}}\dashrightarrow\tilde{\mathbb{P}}'$ relative
to $\tilde{f}:\tilde{\mathbb{P}}\rightarrow\mathbb{P}^{1}$. For this,
it suffices to show that the restriction $\sigma_{6}\mid_{E_{6}}:E_{6}\rightarrow L_{5}$
is isomorphic to the trivial $\mathbb{P}^{1}$-bundle $\mathrm{pr}_{2}:E_{6}\simeq\mathbb{P}^{1}\times\mathbb{P}^{1}\rightarrow L_{5}\simeq\mathbb{P}^{1}$
and that the class of a fiber of the first projection generates an
extremal ray $R$ of the closure $\overline{NE}(\tilde{\mathbb{P}}/\mathbb{P}^{1})$
of the relative cone of curves of $\tilde{\mathbb{P}}$ over $\mathbb{P}^{1}$.
Indeed, if so, the contraction $\varphi:\tilde{\mathbb{P}}\rightarrow\tilde{\mathbb{P}}_{1}$
associated to this extremal ray is the first step of a MMP relative
to the morphism $\tilde{f}:\tilde{\mathbb{P}}\rightarrow\mathbb{P}^{1}$
consisting of the divisorial contraction of $E_{6}$ onto a smooth
curve isomorphic to $\mathbb{P}^{1}$. The existence of $R$ is an
immediate consequence of the following lemma which completes the proof
of Proposition \ref{prop:dp4-cylinder}. 
\begin{lem}
The pair $(E_{6},\mathcal{N}_{E_{6}/\tilde{\mathbb{P}}})$ is isomorphic
to $(\mathbb{P}^{1}\times\mathbb{P}^{1},\mathcal{O}_{\mathbb{P}^{1}\times\mathbb{P}^{1}}(-1,-1))$. \end{lem}
\begin{proof}
The normal bundle $\mathcal{N}_{S/\mathbb{P}^{3}}$ of $S$ in $\mathbb{P}^{3}$
is isomorphic to $\mathcal{O}_{S}(3(C+L))$, and by construction of
the resolution $\sigma:\tilde{\mathbb{P}}=\mathbb{P}_{6}\rightarrow\mathbb{P}^{3}$,
it follows that $\mathcal{N}_{S_{i}/\mathbb{P}_{i}}\simeq\mathcal{O}_{S_{i}}((3-i)C_{i}+3L_{i})$,
$i=1,2,3$, $\mathcal{N}_{S_{4}/\mathbb{P}_{4}}=\mathcal{O}_{S_{4}}(2L_{4})$
and $\mathcal{N}_{S_{5}/\mathbb{P}_{5}}=\mathcal{O}_{S_{5}}(L_{5})$.
Since $L_{5}$ is a line in the cubic surface $S_{5}$, we have $\mathcal{N}_{L_{5}/S_{5}}\simeq\mathcal{O}_{\mathbb{P}^{1}}(-1)$
and $\mathcal{N}_{S_{5}/\mathbb{P}_{5}}\mid_{L_{5}}\simeq\mathcal{O}_{\mathbb{P}^{1}}(L_{5}^{2})\simeq\mathcal{O}_{\mathbb{P}^{1}}(-1)$,
where $L_{5}^{2}$ denotes the self-intersection of $L_{5}$ on the
surface $S_{5}$. It then follows from the exact sequence 
\[
0\rightarrow\mathcal{N}_{L_{5}/S_{5}}\rightarrow\mathcal{N}_{L_{5}/\mathbb{P}_{5}}\rightarrow\mathcal{N}_{S_{5}/\mathbb{P}_{5}}\mid_{L_{5}}\rightarrow0
\]
and the vanishing of $\mathrm{Ext}^{1}(\mathcal{O}_{\mathbb{P}^{1}}(-1),\mathcal{O}_{\mathbb{P}^{1}}(-1))\simeq H^{1}(\mathbb{P}^{1},\mathcal{O}_{\mathbb{P}^{1}})$,
that $\mathcal{N}_{L_{5}/P_{5}}\simeq\mathcal{O}_{\mathbb{P}^{1}}(-1)\oplus\mathcal{O}_{\mathbb{P}^{1}}(-1)$.
Thus $E_{6}\simeq\mathbb{P}(\mathcal{N}_{L_{5}/\mathbb{P}_{5}})$
is isomorphic to $\mathbb{P}^{1}\times\mathbb{P}^{1}$. Since $K_{\mathbb{P}^{1}\times\mathbb{P}^{1}}$
is of type $(-2,-2)$ in the Picard group of $\mathbb{P}^{1}\times\mathbb{P}^{1}$,
it is enough to show that $K_{E_{6}}=2E_{6}\mid_{E_{6}}$. Let $L_{0}$
be a fiber of the restriction of $\sigma_{6}\mid_{E_{6}}:E_{6}\rightarrow L_{5}$
of $\sigma_{6}$. Since $K_{\mathbb{P}_{6}}=\sigma_{6}^{*}K_{\mathbb{P}_{5}}+E_{6}$,
it follows from the adjunction formula that 
\[
K_{E_{6}}=(K_{\mathbb{P}_{6}}+E_{6})\mid_{E_{6}}=(K_{\mathbb{P}_{5}}\cdot L_{5})L_{0}+2E_{6}\mid_{E_{6}}.
\]
On the other hand, since $L_{5}$ is a $(-1)$-curve on $S_{5}$ and
$\mathcal{N}_{S_{5}/\mathbb{P}_{5}}\mid_{L_{5}}\simeq\mathcal{O}_{\mathbb{P}^{1}}(-1)$,
we have 
\[
1=-K_{S_{5}}\cdot L_{5}=(-(K_{\mathbb{P}_{5}}+S_{5})\mid_{S_{5}}\cdot L_{5})=-K_{\mathbb{P}_{5}}\cdot L_{5}+1.
\]
Thus $K_{\mathbb{P}_{5}}\cdot L_{5}=0$, and hence $K_{E_{6}}=2E_{6}\mid_{E_{6}}$
as desired.  
\end{proof}
\bibliographystyle{amsplain}

\end{document}